\documentclass[12pt,reqno]{amsart}
\usepackage{amsmath,amssymb,amsfonts,amsthm}
\usepackage{geometry}
\usepackage{hyperref}
\usepackage{color}

\newtheorem{theorem}{Theorem}
\newtheorem{lemma}[theorem]{Lemma}
\newtheorem{proposition}[theorem]{Proposition}
\newtheorem{corollary}[theorem]{Corollary}
\newtheorem{remark}[theorem]{Remark}

\numberwithin{equation}{section}

\newcommand{\Hsp}{\mathcal{H}}
\newcommand{\B}{\mathcal{B}(\mathcal H)}
\newcommand{\norm}[1]{\left\lVert #1\right\rVert}
\newcommand{\ip}[2]{\left\langle #1,#2\right\rangle}
\newcommand{\R}{\mathbb{R}}
\newcommand{\C}{\mathbb{C}}
\newcommand{\wcart}[1]{\omega_{\mathrm{Cart},#1}}

\begin{document}

\title{The \(L^2\)-Cartesian Numerical Radii of Hilbert Space Operators}
\author{Wasim Audeh}
\address{Department of Mathematics, University of Petra, Amman, Jordan}
\email{waudeh@uop.edu.jo}

\author{Fuad Kittaneh}
\address{Department of Mathematics, The University of Jordan, Amman, Jordan}
\address{Department of Mathematics, Korea University, 02841, Seoul, South Korea}
\email{fkitt@ju.edu.jo}

\author{Hamid Reza Moradi}
\address{Department of Mathematics, Ma.C., Islamic Azad University, Mashhad, Iran}
\email{hrmoradi@mshdiau.ac.ir}

\author{Satyajit Sahoo}
\address{Department of Mathematics, School of Basic Sciences, Indian Institute of Technology Bhubaneswar, Bhubaneswar, Odisha 752050, India}
\email{satyajitsahoo2010@gmail.com}

\subjclass[2010]{Primary 47A12; Secondary 47A30, 47A63, 47A05, 47B15, 46C05}
\keywords{\(L^2\)-Cartesian numerical radius; numerical radius; Cartesian decomposition; sharp operator inequality}

\begin{abstract}
Let \(\mathcal{H}\) be a complex Hilbert space, and let \(T=H+iK\in\mathcal{B}(\mathcal{H})\) be the Cartesian decomposition of \(T\), where \(H=H^{*}\) and \(K=K^{*}\). We introduce the \(L^2\)-Cartesian numerical radius and study its basic properties. We compare it with the classical numerical radius \(\omega(T)\).
In particular, we establish the sharp estimate
\[
\frac{1}{\sqrt{2}}\omega(T)\le \wcart{2}(T)\le \omega(T),
\]
and establish lower and upper bounds in terms of the Cartesian parts \(H,K\).
We also give an equality characterization and explicit sharpness examples.
As a consequence, we obtain
\[
\frac{1}{4}\left\|T^{*}T+TT^{*}\right\|\le \wcart{2}(T)^2\le \omega(T)^2.
\]
More generally, we study the \(L^p\)-Cartesian numerical radius and prove the sharp comparison
\[
c_p\,\omega(T)\le \wcart{p}(T)\le \omega(T),\qquad 1\le p<\infty.
\]
Finally, we present a block-operator analogue of the classical refinement of the triangle inequality.
\end{abstract}

\maketitle

\section{Introduction}

We use standard notation: \(\mathcal B(\mathcal H)\) is the algebra of all bounded linear operators on \(\mathcal H\), equipped with the operator norm
\[
\|T\|=\sup_{\|x\|=1}\|Tx\|.
\]
We denote the zero operator by \(O\) and the identity operator by \(I\). The adjoint of \(T\) is denoted by \(T^*\), the inner product by \(\langle\cdot,\cdot\rangle\), and a vector \(x\in\mathcal H\) is unit if \(\|x\|=1\). An operator \(U\) is unitary if \(U^*U=UU^*=I\). Two operators \(A,B\) commute if \(AB=BA\), and \(T\) is normal if \(T^*T=TT^*\). For \(T\in\mathcal B(\mathcal H)\), its real and imaginary parts are
\[
H=\Re T=\frac{T+T^*}{2},\qquad K=\Im T=\frac{T-T^*}{2i}.
\]
Write the Cartesian decomposition of \(T\) as
\[
T=H+iK,\qquad H=H^{*},\qquad K=K^{*}.
\]
Recall that the numerical radius of \(T\in\mathcal B(\mathcal H)\) is
\[
\omega(T)=\sup_{\|x\|=1}|\langle Tx,x\rangle|.
\]

The numerical radius \(\omega(T)\) admits the Cartesian characterization
\[
\omega(T)=\sup_{\theta\in\mathbb{R}}\|\cos\theta H+\sin\theta K\|.
\]
For \(T=H+iK\in\B\) and \(\theta\in\mathbb R\), we write
\[
A_\theta=\cos\theta H+\sin\theta K.
\]
For \(1\le p<\infty\), the \emph{\(L^p\)-Cartesian numerical radius} of \(T\) is defined by
\begin{equation}\label{eq:def-omega-p}
\wcart{p}(T)=\left(\frac{1}{2\pi}\int_0^{2\pi}\|A_\theta\|^p\,d\theta\right)^{1/p},
\end{equation}
and we set \(\wcart{\infty}(T)=\omega(T)\). The case \(p=2\) is the \(L^2\)-Cartesian numerical radius
\[
\wcart{2}(T)=\left(\frac{1}{2\pi}\int_0^{2\pi}\|A_\theta\|^2\,d\theta\right)^{1/2}.
\]
We also use the constant
\begin{equation}\label{eq:def-cp}
c_p=\left(\frac{1}{2\pi}\int_0^{2\pi}|\cos\theta|^p\,d\theta\right)^{1/p}\in(0,1],
\end{equation}
so that \(c_2=1/\sqrt2\), \(c_1=2/\pi\), and \(\lim_{p\to\infty}c_p=1\).

For a unit vector \(x\in\Hsp\), we write
\[
h_x=\ip{Hx}{x},\qquad k_x=\ip{Kx}{x},
\]
and set
\[
V_T(x)=\norm{(H-h_xI)x}^2+\norm{(K-k_xI)x}^2.
\]
For \(T,S\in\B\), we denote the off-diagonal block operator acting on the direct sum Hilbert space \(\Hsp\oplus\Hsp\) by
\[
\mathbb M_{T,S}=
\begin{bmatrix}
O & T\\
S^* & O
\end{bmatrix}.
\]

This block representation, together with the Cartesian decomposition, has been a fundamental tool in studying numerical radius inequalities. For instance, Kittaneh, Moslehian, and Yamazaki \cite{KittanehMoslehianYamazaki2015} utilized it to obtain sharp refinements of the triangle inequality and several numerical radius bounds. Kittaneh \cite{Kittaneh2005} showed the classical estimate
\begin{equation}\label{int1}
\omega(T)^2\le \frac{1}{2}\|T^*T+TT^*\|=\|H^2+K^2\|,
\end{equation}
while further sharp bounds involving the Cartesian parts were provided in \cite{Kittaneh2015}. More recent works have continued this line; see, e.g., Nourbakhsh et al. \cite{NourbakhshHassaniOmidvarMoradi2024} for improved inequalities via Cartesian decomposition, Alomari et al. \cite{AlomariMoradiSababheh2026, AlomariSababhehCondeMoradi2024} for extensions of the Euclidean operator radius, and Barik et al. \cite{BarikBhuniaPaul2025} for a new notion of Cartesian decomposition and \(P\)-numerical radius bounds. Related refinements and off-diagonal block operator estimates appear in \cite{HirzallahKittanehShebrawi2011, SababhehCondeMoradi2023}; see also the survey by Dragomir \cite{Dragomir2007}.

The literature on numerical radius inequalities is extensive and growing rapidly. To see related works on the broader theory of numerical ranges and radii, see \cite{BonsallDuncan1971, GustafsonRao1997}. For new and updated work on numerical radius inequalities via Cartesian decomposition, see \cite{ElHaddadKittaneh2007, SababhehCondeMoradi2023}. For recent studies on numerical radius inequalities, see, for example, \cite{AudehAlLabadi2023, AudehAlLabadiAlNaimi2025, AudehMoradiSababheh2024a, AudehMoradiSababheh2025a, AudehMoradiSababheh2025b, AudehMoradiSababheh2026a, AudehMoradiSababheh2026b, AudehMoradiSababheh2026c, AudehMoradiSababheh2026d}. Related concepts such as the \(p\)-numerical radius have been established in \cite{BenmakhloufFrakisKittanehMennouni2025, BenmakhloufHirzallahKittaneh2021, BhuniaSahoo2025, FrakisKittaneh2025, FrakisKittanehSoltani2024a, FrakisKittanehSoltani2024b, FrakisKittanehSoltani2025, Sahoo2026}, the \(\rho\)-operator radius is studied in \cite{BhuniaMoslehianZamani2026, KittanehZamani2024}, and the \(q\)-numerical radius is presented in \cite{KittanehRashid2025, MoghaddamMirmostafaeeJanfada2024, SahooRout2026, StankovicKrsticDamnjanovic2025}. For weighted and \(C^{*}\)-algebraic settings, see \cite{ChenLiangLong2026}.

Motivated by these developments, we study the \(L^2\)-Cartesian numerical radius \(\wcart{2}\), which is the case \(p=2\) of \eqref{eq:def-omega-p}. This quantity is a norm on \(\mathcal B(\mathcal H)\), as shown in Theorem~\ref{thm:norm}. It is defined through an integral and is therefore amenable to Hilbert-space averaging, convexity, and spectral arguments.

One of the main results of the paper is the sharp comparison
\[
\frac{1}{\sqrt{2}}\omega(T)\le \wcart{2}(T)\le \omega(T),
\]
valid for every \(T\in\mathcal{B}(\mathcal{H})\). Both constants are sharp. The lower constant is attained by nonzero self-adjoint operators, while the upper constant is attained by operators whose Cartesian parts satisfy \(H^2=K^2=I\) and \(HK+KH=O\). This shows that \(\wcart{2}(\cdot)\) is equivalent to, but generally different from, the classical numerical radius. For related comparisons between different radii and norms, see \cite{Dragomir2006, HeydarbeygiSababhehMoradi2022, KittanehMoradiSababheh2024, MoradiSababheh2021, Yamazaki2007}.

We also obtain lower and upper bounds for \(\wcart{2}(\cdot)\) in terms of the Cartesian parts. For instance,
\[
\wcart{2}(T)^2\ge \frac{1}{2}\|H^2+K^2\|=\frac{1}{4}\|T^*T+TT^*\|.
\]
Moreover, we clarify a vector-dependent reformulation of this lower bound. For other lower bounds involving the real and imaginary parts, see \cite{NourbakhshHassaniOmidvarMoradi2024}. We establish the sharp two-sided estimate
\[
\|T+S\|\le 2\sqrt{2}\,\wcart{2}(\mathbb M_{T,S})
\le \sqrt{2}\bigl(\|T\|+\|S\|\bigr),
\]
where both constants are sharp. This result is the \(L^2\)-Cartesian numerical radius analogue of the classical refinement of the triangle inequality. It should not be viewed as a refinement of the ordinary triangle inequality, since its right-hand side carries the factor \(\sqrt{2}\). For the classical numerical radius version of this refinement, see \cite[Theorem 2.3]{KittanehMoslehianYamazaki2015}.

Finally, we show in Section~\ref{sec:Lp} that the sharp comparison above admits a natural one-parameter extension: for every \(1\le p<\infty\), the \(L^p\)-Cartesian numerical radius \(\wcart{p}(T)\) defined in \eqref{eq:def-omega-p} satisfies
\[
c_p\,\omega(T)\le \wcart{p}(T)\le \omega(T),
\]
where \(c_p\) is defined in \eqref{eq:def-cp}. In this sense, the \(L^2\)-Cartesian numerical radius studied in Section~\ref{s2} is the basic member of a unified family that interpolates between an \(L^1\)-average and the classical numerical radius.

\section{The \(L^2\)-Cartesian numerical radius}\label{s2}

We start this section with some basic properties of \(\wcart{2}(\cdot)\).

\begin{theorem}\label{thm:norm}
The functional \(T\mapsto \wcart{2}(T)\) from \(\B\) to \([0,\infty)\) is a norm. More precisely, it enjoys the following properties:
\begin{itemize}
\item[(i)] \(\wcart{2}(T)\geq0\), \(\wcart{2}(T)=0\iff T=O\).
\item[(ii)] \(\wcart{2}(\lambda T)=|\lambda|\wcart{2}(T)\), \(\lambda\in\C\).
\item[(iii)] \(\wcart{2}(S+T)\leq\wcart{2}(S)+\wcart{2}(T)\).
\item[(iv)] \(\wcart{2}(U^*TU)=\wcart{2}(T)\) for every unitary operator \(U\).
\item[(v)] \(\wcart{2}(e^{i\phi}T)=\wcart{2}(T)\) and \(\wcart{2}(T^*)=\wcart{2}(T)\) for every \(\phi\in\R\).
\end{itemize}
\end{theorem}

\begin{proof}
(i) Nonnegativity is immediate. Suppose \(\wcart{2}(T)=0\). The function
\[
\theta\longmapsto
\norm{\cos\theta H+\sin\theta K}^2
\]
is continuous and nonnegative, and its integral over \([0,2\pi]\) is zero. Hence, it must vanish identically. Evaluating at
\(\theta=0\) and \(\theta=\pi/2\) gives \(H=O\) and \(K=O\), hence, \(T=O\).

(ii) If \(\lambda=0\), then \(\lambda T=O\) and both sides are \(0\). Assume
\(\lambda\ne0\). Write \(\lambda=re^{i\phi}\), where \(r=|\lambda|>0\). The
Cartesian components of \(\lambda T\) are
\[
H_\lambda=r(\cos\phi\,H-\sin\phi\,K),
\qquad
K_\lambda=r(\sin\phi\,H+\cos\phi\,K).
\]
Therefore,
\[
\cos\theta H_\lambda+\sin\theta K_\lambda
=
r\bigl(\cos(\theta-\phi)H+\sin(\theta-\phi)K\bigr).
\]
After the change of variable \(u=\theta-\phi\),
\[
\wcart{2}(\lambda T)^2
=r^2\wcart{2}(T)^2,
\]
and hence,
\[
\wcart{2}(\lambda T)=|\lambda|\wcart{2}(T).
\]

(iii) If
\[
S=H_S+iK_S,\qquad T=H_T+iK_T,
\]
then
\[
\begin{aligned}
&\norm{\cos\theta(H_S+H_T)+
\sin\theta(K_S+K_T)}\\
&\leq
\norm{\cos\theta H_S+\sin\theta K_S}
+
\norm{\cos\theta H_T+\sin\theta K_T}.
\end{aligned}
\]
Let
\[
f_S(\theta)=\|\cos\theta H_S+\sin\theta K_S\|,\qquad
f_T(\theta)=\|\cos\theta H_T+\sin\theta K_T\|,
\]
and define \(f_{S+T}\) analogously by
\[
f_{S+T}(\theta)=\|\cos\theta(H_S+H_T)+\sin\theta(K_S+K_T)\|.
\]
Then the preceding inequality says
\[
f_{S+T}(\theta)\le f_S(\theta)+f_T(\theta)\qquad(\theta\in[0,2\pi]).
\]
Taking \(L^2\)-norms with respect to the normalized measure
\(d\theta/(2\pi)\) and applying Minkowski's inequality, we obtain
\[
\wcart{2}(S+T)
=\|f_{S+T}\|_{L^2}
\le \|f_S+f_T\|_{L^2}
\le \|f_S\|_{L^2}+\|f_T\|_{L^2}
=\wcart{2}(S)+\wcart{2}(T).
\]

(iv) The Cartesian decomposition of \(U^*TU\) is
\[
U^*TU=U^*HU+iU^*KU.
\]
Thus,
\[
\begin{aligned}
\wcart{2}(U^*TU)^2
&=
\frac1{2\pi}\int_0^{2\pi}
\norm{U^*(\cos\theta H+\sin\theta K)U}^2\,d\theta\\
&=
\frac1{2\pi}\int_0^{2\pi}
\norm{\cos\theta H+\sin\theta K}^2\,d\theta\\
&=\wcart{2}(T)^2.
\end{aligned}
\]

(v) The first assertion follows from the homogeneity proof above with
\(|\lambda|=1\). For the second, the Cartesian decomposition of \(T^*\) is
\(H-iK\), and hence,
\[
\wcart{2}(T^*)^2
=
\frac1{2\pi}\int_0^{2\pi}
\norm{\cos\theta H-\sin\theta K}^2\,d\theta.
\]
The substitution \(\theta\mapsto-\theta\) proves the assertion.
\end{proof}

The following two propositions describe cases in which \(\wcart{2}\) is explicitly computable.

\begin{proposition}\label{prop:selfadjoint}
Let \(T\in\B\). If \(T=T^*\), then
\[
\wcart{2}(T)=\frac1{\sqrt2}\norm{T} =\frac1{\sqrt2}\omega(T).
\]
\end{proposition}

\begin{proof}
In this case \(K=O\), so
\[
\wcart{2}(T)^2
=
\frac1{2\pi}\int_0^{2\pi}
\cos^2\theta\,d\theta\,\norm{T}^2
=
\frac12\norm{T}^2.
\]
Since \(\omega(T)=\norm{T}\) for self-adjoint operators, the result follows.
\end{proof}

\begin{proposition}\label{prop:equalitycase}
Let \(T=H+iK\in\B\). If
\[
H^2=K^2
\qquad\text{and}\qquad
HK+KH=O,
\]
then
\[
\omega(T)=\wcart{2}(T)=\sqrt{\frac12\|H^2+K^2\|}.
\]
\end{proposition}

\begin{proof}
Under the assumptions, we have
\[
H^2-K^2=O
\]
and
\[
HK+KH=O.
\]
Therefore,
\[
A_\theta^2
=
\frac12(H^2+K^2)
\]
for every \(\theta\). Since \(A_\theta\) is self-adjoint,
\[
\|A_\theta\|^2=\|A_\theta^2\|.
\]
Hence,
\[
\norm{A_\theta}^2
=
\frac12\norm{H^2+K^2}
\]
for every \(\theta\). Thus,
\[
\omega(T)^2
=
\sup_\theta\norm{A_\theta}^2
=
\frac12\norm{H^2+K^2}.
\]
Also,
\[
\wcart{2}(T)^2
=
\frac{1}{2\pi}
\int_0^{2\pi}\norm{A_\theta}^2\,d\theta
=
\frac12\norm{H^2+K^2}.
\]
Therefore,
\[
\omega(T)^2=\wcart{2}(T)^2.
\]
Taking square roots proves the result.
\end{proof}

We now present one of the main results of this paper.

\begin{theorem}\label{thm:comparison}
Let \(T\in\B\). Then
\[
\frac1{\sqrt2}\omega(T) \leq \wcart{2}(T) \leq \omega(T).
\]
Both constants are sharp.
\end{theorem}

\begin{proof}
By the Cartesian characterization of \(\omega(T)\), we have
\[
\norm{\cos\theta H+\sin\theta K}\leq \omega(T)
\]
for every \(\theta\). Hence,
\[
\wcart{2}(T)^2
\leq
\frac1{2\pi}\int_0^{2\pi}\omega(T)^2\,d\theta
=\omega(T)^2.
\]
Therefore, \(\wcart{2}(T)\leq \omega(T)\).

For the lower bound, let \(x\in\mathcal H\) be a unit vector and put
\[
a=\ip{Hx}{x},\qquad b=\ip{Kx}{x}.
\]
Then
\[
\ip{Tx}{x}=a+ib.
\]
Consequently,
\[
\begin{aligned}
\wcart{2}(T)^2
&\geq
\frac1{2\pi}\int_0^{2\pi}
|a\cos\theta+b\sin\theta|^2\,d\theta\\
&=
\frac12(a^2+b^2)\\
&=\frac12|\ip{Tx}{x}|^2.
\end{aligned}
\]
Taking the supremum over all unit vectors \(x\) yields
\[
\wcart{2}(T)^2\geq\frac12\omega(T)^2,
\]
and hence,
\[
\wcart{2}(T)\geq\frac1{\sqrt2}\omega(T).
\]

To see sharpness of the lower constant, let \(T=T^*\neq O\). Then \(K=O\), and according to Proposition~\ref{prop:selfadjoint},
\[
\wcart{2}(T)=\frac1{\sqrt2}\omega(T).
\]
The upper constant \(1\) is also sharp. Let
\[
H=
\begin{bmatrix}0&1\\1&0\end{bmatrix},
\qquad
K=
\begin{bmatrix}1&0\\0&-1\end{bmatrix},
\qquad T=H+iK.
\]
Then \(H^2=K^2=I\) and \(HK+KH=0\). Hence, Proposition~\ref{prop:equalitycase} gives
\[
\wcart{2}(T)=\omega(T)=1.
\]
\end{proof}

\begin{proposition}
\label{prop:equalitychar}
Let \(T=H+iK\in\B\). Then \(\wcart{2}(T)=\omega(T)\) if and only if the function
\[
\theta\longmapsto \|\cos\theta H+\sin\theta K\|
\]
is constant on \([0,2\pi]\).
\end{proposition}

\begin{proof}
Put \(f(\theta)=\|\cos\theta H+\sin\theta K\|\). Then \(f\) is continuous and nonnegative,
\[
\omega(T)=\sup_{\theta\in[0,2\pi]} f(\theta),
\qquad
\wcart{2}(T)^2=\frac1{2\pi}\int_0^{2\pi} f(\theta)^2\,d\theta.
\]
If \(f\) is constant, then clearly \(\wcart{2}(T)=\omega(T)\). Conversely, if
\(\wcart{2}(T)=\omega(T)\), then the average of \(f^2\) equals its maximum. By continuity, \(f^2\) must be constant, hence \(f\) is constant.
\end{proof}

The following result includes a lower bound for \(\wcart{2}(\cdot)\).

\begin{theorem}\label{thm:lower}
Let \(T=H+iK\in\B\). Then
\[
\wcart{2}(T)^2 \geq \frac12\norm{H^2+K^2}.
\]
Equivalently,
\[
\wcart{2}(T)^2 \geq \frac14\norm{T^*T+TT^*}.
\]
\end{theorem}

\begin{proof}
Since \(A_\theta=A_\theta^*\),
\[
\norm{A_\theta}^2=\norm{A_\theta^2}.
\]
By convexity of the operator norm, we have
\[
\begin{aligned}
\wcart{2}(T)^2
&=
\frac1{2\pi}\int_0^{2\pi}\norm{A_\theta^2}\,d\theta\\
&\geq
\norm{
\frac1{2\pi}\int_0^{2\pi}A_\theta^2\,d\theta}.
\end{aligned}
\]
Now,
\[
A_\theta^2
=
\cos^2\theta H^2
+\sin^2\theta K^2
+\sin\theta\cos\theta(HK+KH).
\]
Using
\[
\frac1{2\pi}\int_0^{2\pi}\cos^2\theta\,d\theta
=
\frac12,
\quad
\frac1{2\pi}\int_0^{2\pi}\sin^2\theta\,d\theta
=
\frac12,
\]
and
\[
\frac1{2\pi}\int_0^{2\pi}\sin\theta\cos\theta\,d\theta=0,
\]
we obtain
\[
\frac1{2\pi}\int_0^{2\pi}A_\theta^2\,d\theta
=
\frac12(H^2+K^2).
\]
Thus,
\[
\wcart{2}(T)^2\geq\frac12\norm{H^2+K^2}.
\]
Since
\[
H^2+K^2
=
\frac{T^*T+TT^*}{2},
\]
the second inequality follows.
\end{proof}

\begin{corollary}
Let \(T\in\B\). Then
\[
\frac14\norm{T^*T+TT^*} \leq \wcart{2}(T)^2 \leq \omega(T)^2.
\]
\end{corollary}

\begin{remark}
The lower bound is sharp for self-adjoint operators. Indeed, if \(T=T^*\), then \(K=O\), and Proposition~\ref{prop:selfadjoint} gives
\[
\wcart{2}(T)^2=\frac12\|T\|^2=\frac12\|H\|^2.
\]
\end{remark}

\begin{theorem}
\label{thm:equality-compact}
Let \(T\in\B\) be compact, and write its Cartesian decomposition as
\(T=H+iK\). Then the following statements are equivalent:
\begin{enumerate}
\item
\[
\wcart{2}(T)^2=\frac12\norm{H^2+K^2}.
\]
\item There exists a unit vector \(x\in\Hsp\) such that
\[
\norm{A_\theta x}=\norm{A_\theta}
\qquad\text{for every }\theta\in\R.
\]
That is, \(x\) is a common unit norming vector for the family \(\{A_\theta\}_{\theta\in\mathbb R}\).
\end{enumerate}
\end{theorem}

\begin{proof}
Set \(Q=\frac12(H^2+K^2)\). Since \(T\) is compact, both \(H\) and \(K\)
are compact, and hence \(Q\) is a compact positive operator. Moreover,
\[
\frac1{2\pi}\int_0^{2\pi}A_\theta^2\,d\theta=Q,
\qquad
\wcart{2}(T)^2=\frac1{2\pi}\int_0^{2\pi}\norm{A_\theta^2}\,d\theta.
\]
Suppose first that \(\wcart{2}(T)^2=\norm Q\). Since \(Q\) is compact and
positive, there is a unit vector \(x\) such that \(\ip{Qx}{x}=\norm Q\). It follows that
\[
\begin{aligned}
0
&=\wcart{2}(T)^2-\ip{Qx}{x}\\
&=\frac1{2\pi}\int_0^{2\pi}
\bigl(\norm{A_\theta^2}-\ip{A_\theta^2x}{x}\bigr)\,d\theta.
\end{aligned}
\]
For every \(\theta\), the integrand is nonnegative because \(A_\theta^2\)
is positive and
\(\ip{A_\theta^2x}{x}\leq\norm{A_\theta^2}\). It is also continuous.
Thus, it vanishes for every \(\theta\). As \(A_\theta\) is self-adjoint,
\[
\norm{A_\theta^2}=\norm{A_\theta}^2,
\qquad
\ip{A_\theta^2x}{x}=\norm{A_\theta x}^2.
\]
Therefore, \(\norm{A_\theta x}=\norm{A_\theta}\) for every \(\theta\).

Conversely, suppose that a unit vector \(x\) satisfies this identity for
every \(\theta\). Then
\[
\begin{aligned}
\wcart{2}(T)^2
&=\frac1{2\pi}\int_0^{2\pi}\norm{A_\theta x}^2\,d\theta\\
&=\frac1{2\pi}\int_0^{2\pi}\ip{A_\theta^2x}{x}\,d\theta\\
&=\ip{Qx}{x}\leq\norm Q.
\end{aligned}
\]
The reverse inequality \(\norm Q\leq\wcart{2}(T)^2\) follows from
\[
\norm Q
=\left\|\frac1{2\pi}\int_0^{2\pi}A_\theta^2\,d\theta\right\|
\leq\frac1{2\pi}\int_0^{2\pi}\norm{A_\theta^2}\,d\theta
=\wcart{2}(T)^2.
\]
Hence, equality holds.
\end{proof}

\begin{remark}
The compactness assumption guarantees that \(Q\) attains its norm at a
unit vector. The theorem characterizes equality by the existence of a
common norming vector for all Cartesian rotations \(A_\theta\).
\end{remark}

The following observation, a pointwise form of the lower bound in Theorem~\ref{thm:lower}, is used in the sequel.

\begin{lemma}\label{lem:variance}
Let \(T=H+iK\in\B\) and let \(x\in\mathcal H\) be a unit vector. Then
\[
\wcart{2}(T)^2\ge \frac12|\ip{Tx}{x}|^2+\frac12 V_T(x)
=\frac12\ip{(H^2+K^2)x}{x}.
\]
\end{lemma}

\begin{proof}
Using \(A_\theta\) from the Introduction, we have
\[
A_\theta x
=
(\cos\theta h_x+\sin\theta k_x)x
+
\cos\theta(H-h_xI)x
+
\sin\theta(K-k_xI)x.
\]
The first term lies in \(\mathrm{span}\{x\}\), while the sum of the
remaining two terms lies in \(x^\perp\), since
\[
\ip{(H-h_xI)x}{x}=0,
\qquad
\ip{(K-k_xI)x}{x}=0.
\]
Hence,
\[
\begin{aligned}
\norm{A_\theta x}^2
={}&
(\cos\theta h_x+\sin\theta k_x)^2\\
&+
\norm{
\cos\theta(H-h_xI)x+
\sin\theta(K-k_xI)x
}^2.
\end{aligned}
\]
Integrating the first term gives
\[
\frac1{2\pi}\int_0^{2\pi}
(\cos\theta h_x+\sin\theta k_x)^2\,d\theta
=
\frac12(h_x^2+k_x^2),
\]
and the cross-product between the two error vectors integrates to zero
because \(\int_0^{2\pi}\sin\theta\cos\theta\,d\theta=0\). Thus,
\[
\begin{aligned}
&\frac1{2\pi}\int_0^{2\pi}\norm{A_\theta x}^2\,d\theta\\
&\qquad=
\frac12(h_x^2+k_x^2)
+
\frac12\norm{(H-h_xI)x}^2
+
\frac12\norm{(K-k_xI)x}^2.
\end{aligned}
\]
Since \(\norm{A_\theta}\geq\norm{A_\theta x}\) and \(h_x^2+k_x^2=|\ip{Tx}{x}|^2\),
\[
\wcart{2}(T)^2
\geq
\frac12|\ip{Tx}{x}|^2+\frac12 V_T(x).
\]
Finally,
\[
V_T(x)
=
\ip{H^2x}{x}-h_x^2+\ip{K^2x}{x}-k_x^2
=
\ip{(H^2+K^2)x}{x}-|\ip{Tx}{x}|^2,
\]
which yields the identity
\[
\frac12|\ip{Tx}{x}|^2+\frac12 V_T(x)
=
\frac12\ip{(H^2+K^2)x}{x}.
\]
\end{proof}

We now sharpen the general lower bound \(\omega(T)^2\le 2\wcart{2}(T)^2\).

\begin{theorem}\label{thm:improved-lower}
Let \(T=H+iK\in\B\). Then
\[
\inf_{\|x\|=1}\max\left\{\|(H-h_xI)x\|^2,\;
\|(K-k_xI)x\|^2\right\}
+\omega(T)^2
\le 2\wcart{2}(T)^2.
\]
\end{theorem}

\begin{proof}
Let \(x\in\mathcal H\) be a unit vector. From Lemma~\ref{lem:variance},
\[
\frac12\max\left\{\|(H-h_xI)x\|^2,\|(K-k_xI)x\|^2\right\}
+\frac12|\ip{Tx}{x}|^2
\le \wcart{2}(T)^2.
\]
Equivalently,
\[
|\ip{Tx}{x}|^2
\le 2\wcart{2}(T)^2
-\max\left\{\|(H-h_xI)x\|^2,\|(K-k_xI)x\|^2\right\}.
\]
Taking the supremum over all unit vectors \(x\) yields
\[
\omega(T)^2
\le 2\wcart{2}(T)^2
-\inf_{\|x\|=1}\max\left\{\|(H-h_xI)x\|^2,\|(K-k_xI)x\|^2\right\}.
\]
Rearranging proves the claim.
\end{proof}

The next result sharpens \eqref{int1}.
\begin{proposition}\label{prop:improved-bound}
Let \(T=H+iK\in\B\). Then
\[
\inf_{\|x\|=1}\max\left\{\|(H-h_xI)x\|^2,\;
\|(K-k_xI)x\|^2\right\}
+\omega(T)^2
\le \frac12\|T^*T+TT^*\|.
\]
\end{proposition}

\begin{proof}
Let \(x\in\mathcal H\) be a unit vector. By Lemma~\ref{lem:variance}, we have the identity
\[
\frac12|\ip{Tx}{x}|^2+\frac12 V_T(x)
=
\frac12\ip{(H^2+K^2)x}{x}.
\]
Since
\[
V_T(x)=\|(H-h_xI)x\|^2+\|(K-k_xI)x\|^2
\ge
\max\left\{\|(H-h_xI)x\|^2,\|(K-k_xI)x\|^2\right\},
\]
we obtain
\[
\frac12\max\left\{\|(H-h_xI)x\|^2,\|(K-k_xI)x\|^2\right\}
+\frac12|\ip{Tx}{x}|^2
\le \frac12\ip{(H^2+K^2)x}{x}
\le \frac12\|H^2+K^2\|.
\]
Thus,
\[
|\ip{Tx}{x}|^2
\le \|H^2+K^2\|
-\max\left\{\|(H-h_xI)x\|^2,\|(K-k_xI)x\|^2\right\}.
\]
Taking the supremum over all unit vectors \(x\),
\[
\omega(T)^2
\le \|H^2+K^2\|
-\inf_{\|x\|=1}\max\left\{\|(H-h_xI)x\|^2,\|(K-k_xI)x\|^2\right\}.
\]
Since \(H^2+K^2=\frac12(T^*T+TT^*)\), rearranging gives the desired inequality.
\end{proof}

\begin{remark}\label{rem:strict-example}
The inequalities in Theorem~\ref{thm:improved-lower} and Proposition~\ref{prop:improved-bound} are generally strict, and the infimum of the maximum can be strictly positive. Indeed, take
\[
T=\begin{bmatrix}0&1\\0&0\end{bmatrix}.
\]
Its Cartesian parts are
\[
H=\frac12\begin{bmatrix}0&1\\1&0\end{bmatrix},\qquad
K=\frac12\begin{bmatrix}0&-i\\ i&0\end{bmatrix}.
\]
For a unit vector \(x=(a,b)^T\in\mathbb C^2\),
\[
h_x=\langle Hx,x\rangle=\Re(a\overline b),\qquad
k_x=\langle Kx,x\rangle=-\Im(a\overline b).
\]
Since \(\|Hx\|^2=\|Kx\|^2=\frac14\),
\[
\|(H-h_xI)x\|^2=\frac14-h_x^2,\qquad
\|(K-k_xI)x\|^2=\frac14-k_x^2.
\]
Also,
\[
h_x^2+k_x^2=|a\overline b|^2\le\frac14,
\]
with equality if and only if \(|a|=|b|=\frac1{\sqrt2}\). Hence,
\[
\max\left\{\frac14-h_x^2,\frac14-k_x^2\right\}
=\frac14-\min\{h_x^2,k_x^2\}
\ge\frac14-\frac18=\frac18,
\]
and equality is attained when \(h_x^2=k_x^2=\frac18\), for instance at
\(x=\frac1{\sqrt2}(1,e^{i\pi/4})^T\). Therefore,
\[
\inf_{\|x\|=1}\max\left\{\|(H-h_xI)x\|^2,\|(K-k_xI)x\|^2\right\}
=\frac18>0.
\]
Moreover, \(\omega(T)=\frac12\), \(\wcart{2}(T)^2=\frac14\), and
\(T^*T+TT^*=I_2\). Consequently,
\[
\frac18+\omega(T)^2=\frac38<\frac12=2\wcart{2}(T)^2,
\]
and
\[
\frac18+\omega(T)^2=\frac38<\frac12=\frac12\|T^*T+TT^*\|.
\]
Thus, both inequalities are strict, with a nonzero max part.
\end{remark}

The upper bound for \(\wcart{2}(T)\) in terms of \(\Re T\) and \(\Im T\) is given in the following theorem.

\begin{theorem}\label{thm:upperomega2}
Let \(T=H+iK\in\B\). Then
\[
\wcart{2}(T)^2\leq\frac12\norm H^2+\frac12\norm K^2+\frac{2}{\pi}\norm H\norm K.
\]
\end{theorem}

\begin{proof}
The result follows from
\[
\norm{\cos\theta H+\sin\theta K}
\leq
|\cos\theta|\norm H+|\sin\theta|\norm K
\]
and the identities
\[
\frac1{2\pi}\int_0^{2\pi}\cos^2\theta\,d\theta
=
\frac12,
\]
\[
\frac1{2\pi}\int_0^{2\pi}\sin^2\theta\,d\theta
=
\frac12,
\]
and
\[
\frac1{2\pi}\int_0^{2\pi}
|\sin\theta\cos\theta|\,d\theta
=
\frac1{\pi}.
\]
\end{proof}

\begin{remark}
The mixed coefficient \(2/\pi\) in Theorem~\ref{thm:upperomega2} is sharp. Let
\[
H=I_2,\qquad
K=\begin{bmatrix}0&1\\1&0\end{bmatrix}.
\]
Then
\[
\cos\theta H+\sin\theta K
=
\begin{bmatrix}\cos\theta&\sin\theta\\ \sin\theta&\cos\theta\end{bmatrix},
\]
whose eigenvalues are \(\cos\theta+\sin\theta\) and \(\cos\theta-\sin\theta\). Hence,
\[
\|\cos\theta H+\sin\theta K\|=|\cos\theta|+|\sin\theta|.
\]
Therefore,
\[
\wcart{2}(T)^2
=
\frac1{2\pi}\int_0^{2\pi}(|\cos\theta|+|\sin\theta|)^2\,d\theta
=
1+\frac{2}{\pi}.
\]
Since \(\|H\|=\|K\|=1\), the right-hand side of Theorem~\ref{thm:upperomega2} is also \(1+2/\pi\). Thus, we cannot reduce the mixed term while keeping the two square-term coefficients at \(1/2\).
\end{remark}

\begin{remark}\label{rem:norm-chain-cart}
Let \(T=H+iK\in\mathcal B(\mathcal H)\). The classical bound
\(\omega(T)\le\norm{H}+\norm{K}\) can be refined in the following way:
\[
\omega(T)\le \sqrt{2}\,\wcart{2}(T)
\le \sqrt{\norm{H}^2+\norm{K}^2+\frac{4}{\pi}\norm{H}\norm{K}}
\le \norm{H}+\norm{K}.
\]
The first inequality follows from \(\frac1{\sqrt2}\omega(T)\le \wcart{2}(T)\),
established in Theorem~\ref{thm:comparison}. For the second,
Theorem~\ref{thm:upperomega2} gives
\[
\wcart{2}(T)^2
\le \frac12\norm{H}^2+\frac12\norm{K}^2+\frac{2}{\pi}\norm{H}\norm{K},
\]
whence
\[
\sqrt2\,\wcart{2}(T)
\le
\sqrt{\norm{H}^2+\norm{K}^2+\frac{4}{\pi}\norm{H}\norm{K}}.
\]
The last inequality is a consequence of \(\frac4\pi<2\), which implies
\[
\norm{H}^2+\norm{K}^2+\frac{4}{\pi}\norm{H}\norm{K}
\le
\norm{H}^2+\norm{K}^2+2\norm{H}\norm{K}
=
(\norm{H}+\norm{K})^2.
\]
Taking square roots yields the last inequality in the chain.
\end{remark}

\begin{theorem}\label{thm:commuting}
Let \(T=H+iK\in\mathcal B(\mathcal H)\). Then
\[
\frac12\|H^2+K^2\|\le \wcart{2}(T)^2\le \|H^2+K^2\|.
\]
Moreover, if \(K=\lambda H\) for some \(\lambda\in\mathbb R\) (or \(H=\lambda K\)), then
\[
\wcart{2}(T)^2=\frac12\|H^2+K^2\|=\frac14\|T^*T+TT^*\|.
\]
\end{theorem}

\begin{proof}
The lower bound is Theorem~\ref{thm:lower}. For the upper bound, use Theorem~\ref{thm:comparison} and the classical estimate \cite{Kittaneh2005}
\[
\omega(T)^2\le \|H^2+K^2\|.
\]
Thus,
\[
\wcart{2}(T)^2\le \omega(T)^2\le \|H^2+K^2\|.
\]

Now, suppose \(K=\lambda H\) for some \(\lambda\in\mathbb R\). If \(\lambda=0\), then \(T=H\) is self-adjoint, and the assertion follows from Proposition~\ref{prop:selfadjoint}. Assume \(\lambda\ne0\). Then
\[
\cos\theta H+\sin\theta K
=(\cos\theta+\lambda\sin\theta)H.
\]
Since \(H\) is self-adjoint,
\[
\|\cos\theta H+\sin\theta K\|
=|\cos\theta+\lambda\sin\theta|\,\|H\|.
\]
Thus,
\[
\wcart{2}(T)^2
=\frac1{2\pi}\int_0^{2\pi}|\cos\theta+\lambda\sin\theta|^2\,d\theta\,\|H\|^2.
\]
Using
\[
\frac1{2\pi}\int_0^{2\pi}\cos^2\theta\,d\theta
=\frac1{2\pi}\int_0^{2\pi}\sin^2\theta\,d\theta
=\frac12,
\]
and
\[
\frac1{2\pi}\int_0^{2\pi}\sin\theta\cos\theta\,d\theta=0,
\]
we obtain
\[
\wcart{2}(T)^2=\frac{1+\lambda^2}{2}\|H\|^2.
\]
On the other hand,
\[
H^2+K^2=(1+\lambda^2)H^2,
\]
so
\[
\|H^2+K^2\|=(1+\lambda^2)\|H\|^2.
\]
Therefore,
\[
\wcart{2}(T)^2=\frac12\|H^2+K^2\|.
\]
Finally, since
\[
H^2+K^2=\frac{T^*T+TT^*}{2},
\]
the equivalent form follows.
\end{proof}

\begin{remark}
Commutativity of \(H\) and \(K\) alone does not imply equality in Theorem~\ref{thm:commuting}. For example, let
\[
T=\begin{bmatrix}1&0\\0&i\end{bmatrix},\qquad
H=\begin{bmatrix}1&0\\0&0\end{bmatrix},\qquad
K=\begin{bmatrix}0&0\\0&1\end{bmatrix}.
\]
Then \(H\) and \(K\) commute, but \(\wcart{2}(T)^2\neq \frac12\|H^2+K^2\|\). Indeed, \(H^2+K^2=I\) and
\[
\|\cos\theta H+\sin\theta K\|
=
\max\{|\cos\theta|,|\sin\theta|\},
\]
so
\[
\wcart{2}(T)^2=\frac12+\frac1\pi\neq \frac12.
\]
\end{remark}

The following theorem gives a sharp two-sided estimate for the norm \(\|T+S\|\) in terms of the \(L^2\)-Cartesian numerical radius of an off-diagonal block operator.

\begin{theorem}\label{thm:block}
Let \(T, S\in\B\). Then
\[
\|T+S\|\leq 2 \sqrt{2}\, \wcart{2}(\mathbb M_{T,S})
\leq \sqrt{2}(\|T\|+\|S\|),
\]
and both constants are sharp. More precisely, the constant \(2\sqrt2\) in the first inequality and the constant \(\sqrt2\) in the second inequality are both sharp; equivalently, the constants \(2\sqrt2\) and \(2\) in the two displayed inequalities
\[
\|T+S\|\le 2\sqrt2\,\wcart{2}(\mathbb M_{T,S}),
\qquad
2\,\wcart{2}(\mathbb M_{T,S})\le \|T\|+\|S\|
\]
are both sharp.
\end{theorem}

\begin{proof}
The classical numerical radius refinement of the triangle inequality, due to Kittaneh, Moslehian, and Yamazaki \cite[Theorem 2.3]{KittanehMoslehianYamazaki2015}, states that
\[
\|T+S\|\le 2\omega(\mathbb M_{T,S})\le \|T\|+\|S\|.
\]
By Theorem~\ref{thm:comparison},
\[
\frac1{\sqrt2}\omega(\mathbb M_{T,S})\le \wcart{2}(\mathbb M_{T,S})\le \omega(\mathbb M_{T,S}).
\]
Hence,
\[
\|T+S\|\le 2\omega(\mathbb M_{T,S})\le 2\sqrt2\,\wcart{2}(\mathbb M_{T,S}),
\]
and
\[
2\sqrt2\,\wcart{2}(\mathbb M_{T,S})\le 2\sqrt2\,\omega(\mathbb M_{T,S})\le \sqrt2(\|T\|+\|S\|).
\]
This proves the desired two-sided estimate.

To see that the left constant is sharp, take \(T=S=1\) on \(\mathbb C\). Then
\[
\mathbb M_{T,S}=
\begin{bmatrix}
0&1\\
1&0
\end{bmatrix},
\]
which is self-adjoint with norm \(1\). By Proposition~\ref{prop:selfadjoint},
\[
\wcart{2}(\mathbb M_{T,S})=\frac1{\sqrt2}.
\]
Also, \(\|T+S\|=2\), and therefore,
\[
2\sqrt2\,\wcart{2}(\mathbb M_{T,S})=2=\|T+S\|.
\]

To see that the right constant \(\sqrt2\) is sharp, take \(T=1\) and \(S=0\) on \(\mathbb C\). Then
\[
\mathbb M_{T,S}=
\begin{bmatrix}
0&1\\
0&0
\end{bmatrix}.
\]
Its Cartesian parts are
\[
H=\frac12\begin{bmatrix}0&1\\1&0\end{bmatrix},\qquad
K=\frac12\begin{bmatrix}0&-i\\i&0\end{bmatrix},
\]
and hence,
\[
A_\theta=\cos\theta H+\sin\theta K
=\frac12\begin{bmatrix}0&\cos\theta-i\sin\theta\\ \cos\theta+i\sin\theta&0\end{bmatrix}.
\]
Thus, \(\|A_\theta\|=1/2\) for every \(\theta\), so \(\wcart{2}(\mathbb M_{T,S})=1/2\). Therefore,
\[
2\sqrt2\,\wcart{2}(\mathbb M_{T,S})=\sqrt2=\sqrt2(\|T\|+\|S\|).
\]
\end{proof}

\begin{remark}
The classical numerical radius version of the inequality in Theorem~\ref{thm:block} is
\[
\|T+S\|\le 2 \omega(\mathbb M_{T,S})\le \|T\|+\|S\|,
\]
for \(T, S\in\B\); see \cite[Theorem 2.3]{KittanehMoslehianYamazaki2015}. View the version above as an analog of this classical refinement, rather than a refinement of the ordinary triangle inequality, because its right-hand side contains the factor \(\sqrt2\).
\end{remark}

\section{The \texorpdfstring{\(L^p\)}{Lp}-Cartesian numerical radius}
\label{sec:Lp}

In this section we show that the sharp comparison of Theorem~\ref{thm:comparison} admits a natural one-parameter extension.
We recall from the Introduction that, for \(1\le p<\infty\), the \(L^p\)-Cartesian numerical radius \(\wcart{p}(T)\) is defined by \eqref{eq:def-omega-p}, and the constant \(c_p\) is defined by \eqref{eq:def-cp}. For \(p=2\), this recovers the quantity \(\wcart{2}\) studied in Section~2, and \(c_2=1/\sqrt2\). Thus, the definition in the Introduction provides a one-parameter family that interpolates between the \(L^1\)-average and the supremum in the Cartesian characterization of the numerical radius.

Throughout this section, \(T=H+iK\in\mathcal B(\mathcal H)\). By the Cartesian characterization of the numerical radius,
\begin{equation}\label{eq:cart}
\omega(T)=\sup_{\theta\in[0,2\pi]}\|A_\theta\|.
\end{equation}

\begin{theorem}\label{thm:Lp-norm}
For every \(1\le p<\infty\), the functional
\(T\mapsto\wcart{p}(T)\) from \(\mathcal B(\mathcal H)\) to \([0,\infty)\) is a norm. More precisely:
\begin{itemize}
\item[(i)] \(\wcart{p}(T)\ge 0\), and \(\wcart{p}(T)=0\) if and only if \(T=O\);
\item[(ii)] \(\wcart{p}(\lambda T)=|\lambda|\,\wcart{p}(T)\) for all \(\lambda\in\mathbb C\);
\item[(iii)] \(\wcart{p}(S+T)\le\wcart{p}(S)+\wcart{p}(T)\);
\item[(iv)] \(\wcart{p}(U^*TU)=\wcart{p}(T)\) for every unitary \(U\in\mathcal B(\mathcal H)\);
\item[(v)] \(\wcart{p}(e^{i\phi}T)=\wcart{p}(T)\) and \(\wcart{p}(T^*)=\wcart{p}(T)\)
for every \(\phi\in\mathbb R\).
\end{itemize}
\end{theorem}

\begin{proof}
For \(p=2\), the result is Theorem~\ref{thm:norm}. For general \(p\), the arguments for (i), (ii), (iv), and (v) are the same as those in Theorem~\ref{thm:norm}, with \(\|\cdot\|_{L^2}\) replaced by \(\|\cdot\|_{L^p}\); we omit repetitions. Only (iii) needs the corresponding Minkowski inequality: with \(f_S,f_T,f_{S+T}\) as in the proof of Theorem~\ref{thm:norm}, we have \(f_{S+T}\le f_S+f_T\), and hence
\[
\wcart{p}(S+T)=\|f_{S+T}\|_{L^p}\le\|f_S+f_T\|_{L^p}\le\|f_S\|_{L^p}+\|f_T\|_{L^p}=\wcart{p}(S)+\wcart{p}(T).
\]
\end{proof}

\begin{theorem}\label{thm:Lp-comparison}
Let \(T\in\mathcal B(\mathcal H)\) and \(1\le p<\infty\). Then
\begin{equation}\label{eq:Lp-comparison}
c_p\,\omega(T)\le\wcart{p}(T)\le\omega(T),
\end{equation}
where \(c_p\) is given by \eqref{eq:def-cp}. Both constants are sharp.
More precisely:
\begin{itemize}
\item[(a)] Equality in the left-hand inequality holds for every nonzero
self-adjoint \(T\);
\item[(b)] Equality in the right-hand inequality holds for
\(T=H+iK\) with \(H^2=K^2=I\) and \(HK+KH=O\).
\end{itemize}
\end{theorem}

\begin{proof}
The upper bound is proved exactly as in Theorem~\ref{thm:comparison}, with \(L^2\) replaced by \(L^p\). For the lower bound, fix a unit vector \(x\in\mathcal H\) and set
\[
a=\langle Hx,x\rangle\in\mathbb R,\qquad b=\langle Kx,x\rangle\in\mathbb R.
\]
As in the proof of Theorem~\ref{thm:comparison}, we have
\[
\|A_\theta\|\ge|a\cos\theta+b\sin\theta|.
\]
If \(r=\sqrt{a^2+b^2}\neq0\), choose \(\varphi\in\mathbb R\) with \((a,b)=r(\cos\varphi,\sin\varphi)\). Then the same change of variable as in Theorem~\ref{thm:comparison} gives
\[
\wcart{p}(T)^p
\ge\frac{1}{2\pi}\int_0^{2\pi}r^p|\cos(\theta-\varphi)|^p\,d\theta
=r^p c_p^p.
\]
Thus, \(\wcart{p}(T)\ge c_p|\langle Tx,x\rangle|\), and taking the supremum over all unit vectors \(x\) yields the lower bound. The sharpness assertions follow as in Theorem~\ref{thm:comparison}: part (a) is the self-adjoint computation with \(|\cos\theta|^p\) replacing \(\cos^2\theta\), and part (b) uses the same pair \(H,K\) for which \(\|A_\theta\|=1\) for all \(\theta\).
\end{proof}

\begin{proposition}\label{prop:cp}
For \(1\le p<\infty\),
\begin{equation}\label{eq:cp-closed}
c_p=\left(\frac{\Gamma\!\left(\frac{p+1}{2}\right)}{\sqrt{\pi}\,\Gamma\!\left(\frac{p+2}{2}\right)}\right)^{1/p},
\end{equation}
where \(\Gamma\) is the gamma function. The map \(p\mapsto c_p\) is continuous
and strictly increasing on \([1,\infty)\), with
\[
c_1=\frac{2}{\pi}\approx 0.6366,\qquad
c_2=\frac{1}{\sqrt 2}\approx 0.7071,\qquad
\lim_{p\to\infty}c_p=1.
\]
In particular, \(c_p<1\) for every \(p<\infty\).
\end{proposition}

\begin{proof}
Using the symmetry of \(|\cos\theta|\) on \([0,2\pi]\),
\[
\frac{1}{2\pi}\int_0^{2\pi}|\cos\theta|^p\,d\theta
=\frac{2}{\pi}\int_0^{\pi/2}\cos^p\theta\,d\theta.
\]
The beta integral gives
\[
\int_0^{\pi/2}\cos^p\theta\,d\theta
=\frac{\sqrt\pi}{2}\cdot\frac{\Gamma\!\left(\frac{p+1}{2}\right)}{\Gamma\!\left(\frac{p+2}{2}\right)},
\]
hence,
\[
\frac{1}{2\pi}\int_0^{2\pi}|\cos\theta|^p\,d\theta
=\frac{1}{\sqrt\pi}\cdot\frac{\Gamma\!\left(\frac{p+1}{2}\right)}{\Gamma\!\left(\frac{p+2}{2}\right)},
\]
which yields \eqref{eq:cp-closed}.

For monotonicity, view \(f(\theta)=|\cos\theta|\) as a function on the
probability space \(([0,2\pi],\mathcal B,d\theta/(2\pi))\). Then
\(c_p=\|f\|_{L^p}\). By Lyapunov's inequality,
\(\|f\|_{L^q}\le\|f\|_{L^p}\) for \(1\le q<p<\infty\), hence, \(c_q\le c_p\).
Since \(f\) is nonconstant, the inequality is strict, so \(p\mapsto c_p\) is
strictly increasing. Continuity is immediate from \eqref{eq:cp-closed}.

The values \(c_1=2/\pi\) and \(c_2=1/\sqrt2\) follow from \eqref{eq:cp-closed}
together with \(\Gamma(1)=1\), \(\Gamma(3/2)=\sqrt\pi/2\), \(\Gamma(2)=1\).
Finally, \(\|f\|_{L^p}\to\|f\|_{L^\infty}=\operatorname{ess\,sup}|\cos\theta|=1\)
as \(p\to\infty\), hence, \(c_p\to 1\).
\end{proof}

\begin{corollary}\label{cor:interpolation}
Let \(T\in\mathcal B(\mathcal H)\) and \(1\le q\le p<\infty\). Then
\[
\wcart{q}(T)\le\wcart{p}(T),
\]
and
\[
\lim_{p\to\infty}\wcart{p}(T)=\omega(T).
\]
\end{corollary}

\begin{proof}
The map \(\theta\mapsto\|A_\theta\|\) is continuous on \([0,2\pi]\), hence,
belongs to \(L^\infty\). On the probability space
\(([0,2\pi],d\theta/(2\pi))\), Lyapunov's inequality gives
\(\wcart{q}(T)\le\wcart{p}(T)\) for \(q\le p\). The limit is the standard
identity \(\lim_{p\to\infty}\|g\|_{L^p}=\|g\|_{L^\infty}\) for
\(g\in L^\infty\), combined with \eqref{eq:cart}.
\end{proof}

\begin{remark}\label{rem:Lp-vs-pnr}
The family \(\{\wcart{p}\}_{1\le p<\infty}\) (together with its limit
\(\wcart{\infty}:=\omega\)) should be distinguished from the
\(p\)-numerical radius studied in
\cite{BenmakhloufFrakisKittanehMennouni2025, BenmakhloufHirzallahKittaneh2021, BhuniaSahoo2025, FrakisKittaneh2025, FrakisKittanehSoltani2024a, FrakisKittanehSoltani2024b, FrakisKittanehSoltani2025, Sahoo2026},
which is defined via the numerical range rather than the Cartesian
decomposition. The quantity \(\wcart{p}(T)\) here involves averaging the
operator norm \(\|A_\theta\|\) over the unit circle, and at \(p=2\) reduces
exactly to the \(L^2\)-Cartesian numerical radius \(\wcart{2}\) of Section~2.
\end{remark}

\begin{remark}\label{rem:Lp-consequences}
Several results of Section~2 are the case \(p=2\) of the results above.
Specifically:
\begin{itemize}
\item Theorem~\ref{thm:comparison} is Theorem~\ref{thm:Lp-comparison} with
\(p=2\) and \(c_2=1/\sqrt2\);
\item Proposition~\ref{prop:selfadjoint} is the \(p=2\) instance of part~(a)
of Theorem~\ref{thm:Lp-comparison};
\item Proposition~\ref{prop:equalitycase} is a more general version of the
\(p=2\) instance of part~(b) of Theorem~\ref{thm:Lp-comparison}: it assumes
only \(H^2=K^2\) and \(HK+KH=O\), without the normalization \(H^2=K^2=I\).
\end{itemize}
Thus, the family \(\{\wcart{p}\}_{1\le p<\infty}\) provides a unified framework
that contains the main comparison theorem of the paper as its most basic
member.
\end{remark}

\subsection*{Declarations}
\begin{itemize}
\item \textbf{Data availability}: Not applicable.

\item \textbf{Declaration of competing interests}: The authors declare that they have no competing interests.

\item \textbf{Funding}: Not applicable.

\item \textbf{Authors' contributions}: The authors contributed equally to this paper. All authors have read and approved this version.
\end{itemize}

\end{document}